\documentclass[a4paper,12pt]{article}
\usepackage{cmap}                        % Поддержка поиска русских слов в PDF (pdflatex)
\usepackage[cp1251]{inputenc}            % Выбор языка и кодировки
\usepackage[english]{babel}
\usepackage[left=2cm,right=2cm,top=2cm,bottom=2cm]{geometry} % поля страницы
\usepackage{amssymb}
\usepackage{amsmath, amsthm, hyperref}
\theoremstyle{plain}
\newtheorem{theorem}{Theorem}
\newtheorem{lemma}{Lemma}
\newtheorem{proposition}{Proposition}
\newtheorem{corollary}{Corollary}[theorem]

\theoremstyle{definition}
\newtheorem{example}{Example}
\newtheorem{remark}{Remark}
\newtheorem{definition}{Definition}
\newtheorem{pr}{Question}

\begin{document}

\title{Arithmetic  graphs and the products of  finite groups\footnote{This work is supported by BFFR $\Phi23\textrm{PH}\Phi\textrm{-}237$
}}
\author{Viachaslau I. Murashka\footnote{e-mail: mvimath@yandex.by}  \footnote{Faculty of Mathematics and Technologies of Programming, Francisk Skorina Gomel State University, Sovetskaya 104, Gomel, 246028,  Belarus}}
\date{}
 \maketitle
\begin{abstract}
The Hawkes   graph $\Gamma_H(G)$ of $G$ is the directed graph whose vertex set coincides with $\pi(G)$ and it has the edge $(p, q)$ whenever $q\in\pi(G/O_{p',p}(G))$. The Sylow graph $\Gamma_s(G)$ of $G$ is the directed graph with vertex set $\pi(G)$ and $(p, q)$ is an edge of $\Gamma_s(G)$ whenever $q \in\pi(N_G(P)/PC_G(P))$ for some Sylow $p$-subgroup $P$ of $G$. The $N$-critical  graph $\Gamma_{Nc}(G)$ of a group $G$ the directed graph whose vertex set coincides with $\pi(G)$ such that $(p, q)$ is an edge of $\Gamma_{Nc}(G)$ whenever $G$ contains a Schmidt $(p, q)$-subgroup, i.e. a Schmidt $\{p, q\}$-subgroup with a normal Sylow $p$-subgroup. In the paper the   Hawkes, the Sylow and the $N$-critical graphs of the products of totally permutable, mutually permutable and $\mathfrak{N}$-connected subgroups are studied.
\end{abstract}

\section{Introduction}

All groups considered are \textbf{finite}, $G$ always denotes a group and $\pi(G)$ is the set of all prime divisors of $|G|$. If a graph has no isolated vertices, then we will define it just by its edges.

There were many papers in which with every  group a certain graph is assigned  and the connection of the geometry of the graph with the properties of the group is studied  since 1878 (for example, see \cite{Abe2000, Cayley1878, Hawkes1968, Kazarin2011, Kondratev1990, Lucchini2009, VM, Williams1981} and etc.). Among such graphs there is an interesting family of arithmetic graphs, i.e. graphs whose vertices are prime divisors of the group's order.

 In 1968 Hawkes \cite{Hawkes1968} introduced the directed graph $\Gamma_H(G)$ of $G$ whose vertex set coincides with $\pi(G)$ and it has the edge $(p, q)$ whenever $q\in\pi(G/O_{p',p}(G))$. This graph has many interesting properties \cite{Hawkes1968,  Vasilev2022, VM}. For example \cite{Hawkes1968}, if it does not have a loop $(p, p)$, then the $p$-length of $G$ is at most 1.

 Recall  \cite{DAniello2007, Kazarin2011} that the Sylow graph $\Gamma_s(G)$ of $G$ is the directed graph with vertex set $\pi(G)$ and $(p, q)$ is an edge of $\Gamma_s(G)$ whenever $q \in\pi(N_G(P)/PC_G(P))$ for some Sylow $p$-subgroup $P$ of $G$. For applications and properties of this graph see  \cite{DAniello2007, Kazarin2011,  Murashka2022, VM}. In particular \cite{Murashka2022}, every connected component of $\Gamma_s(G)$  corresponds to a normal Hall subgroup of $G$.

 Recall that a Schmidt $(p, q)$-group is
a Schmidt group (i.e. non-nilpotent group, all whose proper subgroups are nilpotent) $G$ with $\pi(G) = \{p, q\}$ and a normal Sylow $p$-subgroup. The $N$-critical \cite{VM} graph $\Gamma_{Nc}(G)$ of a group $G$ is the directed graph whose vertex set coincides with $\pi(G)$ such that $(p, q)$ is an edge of $\Gamma_{Nc}(G)$ whenever $G$ contains a Schmidt $(p, q)$-subgroup.
For its properties and applications see \cite{Murashka2021, IMurashka2021, VM}.

Recall that $\mathfrak{N}$ denotes the class of all nilpotent groups.
Following   Carocca \cite{Carocca1996} subgroups $H$ and $K$ are called $\mathfrak{N}$-connected if $\langle x, y\rangle\in\mathfrak{N}$ for every $x\in H$ and $y\in K$. Products of $\mathfrak{N}$-connected subgroups were studied in
\cite{Carocca1996, Francalanci2021, Hauck2003}  and other. Here we prove

\begin{theorem}\label{thm1}
  Let a group $G$ be the product of pairwise permutable and $\mathfrak{N}$-connected subgroups $G_1,\dots, G_n$. Then
  $$ \Gamma_s(G)=\bigcup_{i=1}^n\Gamma_s(G_i), \Gamma_H(G)=\bigcup_{i=1}^n\Gamma_H(G_i), \Gamma_{Nc}(G)=\bigcup_{i=1}^n\Gamma_{Nc}(G_i).$$
\end{theorem}

Recall that $G=AB$ is called a totally permutable product of subgroups $A$ and $B$ if every subgroup of $A$ permutes with every subgroup of $B$. Asaad and Shaalan [4]
proved that a totally permutable product of two supersoluble groups is also supersoluble. This result started a study of totally permutable products in connection with the theory of group's classes (for example, see \cite[Chapter 4]{PFG}).

\begin{theorem}\label{thm2}
  Let a group $G$ be the product of pairwise totally permutable   subgroups $G_1,\dots, G_n$ and $\Gamma(G)=\{(p, q)\mid p, q\in\pi(G), q\in\pi(p-1)\}$. Then

  $$ \Gamma_s(G)\subseteq\bigcup_{i=1}^n\Gamma_s(G_i)\cup \Gamma(G), \Gamma_H(G)\subseteq\bigcup_{i=1}^n\Gamma_H(G_i)\cup \Gamma(G), \Gamma_{Nc}(G)\subseteq\bigcup_{i=1}^n\Gamma_{Nc}(G_i)\cup \Gamma(G).$$
\end{theorem}

\begin{example}
  The symmetric group $S_3$  of degree 3 is a totally permutable product of cyclic groups $Z_3$ and $Z_2$ of orders 3 and 2 respectively. Note that $(3, 2)$ is the unique edge of the Sylow graph, the Hawkes graph and the $N$-critical graph of $S_3$ and the Sylow graphs, the Hawkes graphs and the $N$-critical graphs of $Z_3$ and $Z_2$ have no edges. Thus $\Gamma(S_3)\not\subseteq\Gamma(Z_3)\cup\Gamma(Z_2)$ for $\Gamma\in\{\Gamma_s, \Gamma_H, \Gamma_{Nc}\}$.
\end{example}

\begin{remark}
  Theorems \ref{thm1} and \ref{thm2} follow from a more general result (see Theorem \ref{thm}).
\end{remark}

Recall \cite[Definition 4.1.1]{PFG} that a group $G$ is called a mutually permutable product of its subgroups $A$ and $B$ if $G=AB$,  $A$ permutes with every subgroup of $B$ and    $B$ permutes with every subgroup of $A$. The products of mutually permutable subgroups are widely studied (see \cite[Chapter 4]{PFG}).

\begin{theorem}\label{mut}
Let $G=AB$ be a mutually permutable product of its subgroups $A$ and $B$ and $\Gamma(A, B)=\{(p, q)\mid p\in\pi(A), q\in\pi(B)\cap\pi(p-1)\textrm{ or } p\in\pi(B), q\in\pi(A)\cap\pi(p-1)\}$. Then
$$\Gamma_{Nc}(A)\cup\Gamma_{Nc}(B)\subseteq \Gamma_{Nc}(G)\subseteq \Gamma_{Nc}(A)\cup \Gamma_{Nc}(B)\cup\Gamma(A, B)\textrm{  and }$$
$$\Gamma_{H}(A)\cup\Gamma_{H}(B)\subseteq \Gamma_{H}(G)\subseteq \Gamma_{H}(A)\cup \Gamma_{H}(B)\cup\Gamma(A, B)\cup\{(p,p)\mid p\in\pi(G)\}.$$
\end{theorem}

\begin{example}
  Note that the symmetric group $S_4$  of degree 4 is a mutually permutable product of its Sylow 2-subgroup $P$ and the alternating group $A_4$ of degree 4. Now $E(\Gamma_H(P))=\emptyset, E(\Gamma_H(A_4))=\{(2,3)\}, E(\Gamma(P, A_4))=\{(3,2)\}$ and $E(\Gamma_H(S_4))=\{(2,2), (2,3), (3,2)\}$. Hence $\Gamma_H(S_4)\not\subseteq\Gamma_H(P)\cup\Gamma_H(A_4)\cup\Gamma(P, A_4)$.
\end{example}

Recall \cite{Vasilev} that a formation $\mathfrak{F}$ has the Shemetkov property if every $s$-critical for $\mathfrak{F}$ group is a Schmidt group or a group of prime order. For various properties and applications of such formations see \cite[Chapter 6.4]{BallesterBollinches2006}.

\begin{corollary}\label{Shemetkov} A hereditary formation $\mathfrak{F}$ with the Shemetkov property is closed under taking products of mutually permutable $\mathfrak{F}$-subgroups if and only if it contains all supersoluble Schmidt $\pi(\mathfrak{F})$-groups.     \end{corollary}

\begin{corollary}[{\cite[Theorem 2]{Beidleman2005}}]\label{Beidleman}
  Let $p$ be a prime and $\pi$ be a $p$-special set of primes $($i.e. $q\not\in\pi$ whenever
$p$ divides $q(q-1))$. If $G$ is the mutually
permutable product of two subgroups $A$ and $B$ which are normal extensions of $p$-groups
by $\pi$-groups, the same is true for $G$.
\end{corollary}

\begin{theorem}\label{thm4}
  Let a group $G$ be a product of pairwise mutually permutable   soluble subgroups $G_1,\dots, G_n$ and $\Gamma(G_i, G_j)$ be defined the same way as in Theorem \ref{mut}. Then

  $$ \Gamma_H(G)\subseteq\bigcup_{1\leq i\leq n}\Gamma_H(G_i)\cup\bigcup_{1\leq i, j\leq n, i\neq j} \Gamma(G_i, G_j)\cup\{(p,p)\mid p\in\pi(G)\}\textrm{ and}$$ $$\Gamma_{Nc}(G)\subseteq\bigcup_{1\leq i\leq n}\Gamma_{Nc}(G_i)\cup\bigcup_{1\leq i, j\leq n, i\neq j} \Gamma(G_i, G_j).$$
\end{theorem}

The following result for $n=2$ was proved in \cite[Corollary 7]{Vasilev2017}. 

\begin{corollary}\label{cor41}
Let a group $G$ be a product of pairwise mutually permutable   subgroups $G_1,\dots, G_n$
If every Schmidt subgroup of $G_1, \dots, G_n$ is supersoluble, then every Schmidt subgroup of $G$ is supersoluble.
\end{corollary}

\section{Preliminaries}

Here $\pi(n)$ is the set of prime divisors of $n$; $\pi(\mathfrak{F})=\cup_{G\in\mathfrak{F}}\pi(G)$;   $S_n$ and $A_n$ are the symmetric and the alternating group of degree $n$ respectively; $Z_n$ is the cyclic group of order $n$;    $\Phi(G)$ is the Frattini subgroup of $G$;
 $\mathrm{O}_\pi(G)$ is the greatest normal $\pi$-subgroup of a group $G$ for a set of primes $\pi$. If $\pi=\{p\}$, then $\mathrm{O}_\pi(G)$ is denoted by $\mathrm{O}_p(G)$. If $\pi=\mathbb{P}\setminus\{p\}$, then $\mathrm{O}_\pi(G)$ is denoted by $\mathrm{O}_{p'}(G)$; $\mathrm{O}_{p',p}(G)$ is the greatest normal $p$-nilpotent subgroup of a group $G$. It can be defined by $\mathrm{O}_{p',p}(G)/\mathrm{O}_{p'}(G)=\mathrm{O}_p(G/\mathrm{O}_{p'}(G))$.

Recall  that here a (directed) graph $\Gamma$  is a pair of sets $V(\Gamma)$ and $E(\Gamma)$ where $V(\Gamma)$ is a set of vertices of $\Gamma$ and
$E(\Gamma)$ is a set of edges of $ \Gamma$, i.e. the set of ordered pairs of elements from $V(\Gamma)$.
An edge $(v, v)$ is called a loop. Two graphs $\Gamma_1$ and $\Gamma_2$ are called equal (denoted by $\Gamma_1 = \Gamma_2$) if $V (\Gamma_1) = V (\Gamma_2)$ and $E(\Gamma_1) = E(\Gamma_2)$. Graph
$\Gamma_1$ is called subgraph of $\Gamma_2$ (denoted by $\Gamma_1\subseteq\Gamma_2$) if $V (\Gamma_1) \subseteq V (\Gamma_2)$ and $E(\Gamma_1) \subseteq E(\Gamma_2)$. Graph $\Gamma$ is called
a union of graphs $\Gamma_1$ and $\Gamma_2$ (denoted by $\Gamma = \Gamma_1\cup\Gamma_2)$ if $V(\Gamma) = V (\Gamma_1)\cup V (\Gamma_2)$ and $E(\Gamma) = E(\Gamma_1)\cup E(\Gamma_2)$.

Let $\Gamma\in\{\Gamma_s, \Gamma_H, \Gamma_{Nc}\}$ and $\mathfrak{X}$ be a class of groups. Recall  \cite[Definition 3.1]{VM} that
$$\Gamma(\mathfrak{X})=\bigcup_{G\in\mathfrak{X}}\Gamma(G).$$

\begin{lemma}[{\cite[Theorem 2.7]{VM}}]\label{lem1}
Let $G$ be a group. Then
  \begin{enumerate}
    \item If $\Gamma\in\{\Gamma_H, \Gamma_{Nc}\}$, then $\Gamma(H)\subseteq\Gamma(G)$ for any $H\leq G$.

    \item If $\Gamma\in\{\Gamma_s, \Gamma_H, \Gamma_{Nc}\}$, then $\Gamma(G/N)\subseteq\Gamma(G)$ for any $N\trianglelefteq G$.

    \item If $\Gamma\in\{\Gamma_s, \Gamma_H, \Gamma_{Nc}\}$, then $\Gamma(G/N_1)\cup\Gamma(G/N_2)=\Gamma(G)$ for any $N_1,N_2\trianglelefteq G$ with
       \linebreak $N_1\cap N_2=1$.
       
    \item If $\Gamma\in\{\Gamma_H, \Gamma_{Nc}\}$, then $\Gamma(N_1)\cup\Gamma(N_2)=\Gamma(G)$ for any $N_1,N_2\trianglelefteq G$.
        
        \item If $\Gamma\in\{\Gamma_s, \Gamma_H, \Gamma_{Nc}\}$, then $\Gamma(G_1\times\dots\times G_n)=\Gamma(G_1)\cup\dots\cup\Gamma(G_n)$ for any groups $G_1, \dots,G_n$.  
  \end{enumerate}
\end{lemma}

Let $\mathfrak{X}$ be a class of groups. Recall that a chief factor $H/K$ of  $G$ is called   $\mathfrak{X}$-\emph{central} (see \cite[p. 127--128]{s6}) in $G$, for a class of groups $\mathfrak{X}$,  provided  that   the semidirect product $(H/K)\rtimes (G/C_G(H/K))$ of $H/K$ with $G/C_G(H/K)$ corresponding to the action by conjugation of $G$ on $H/K$ belongs $\mathfrak{X}$. The $\mathfrak{X}$-\emph{hypercenter} $\mathrm{Z}_\mathfrak{X}(G)$  of $G$ is the greatest normal subgroup of $G$ such that every chief factor of $G$ below it is $\mathfrak{X}$-central (it exists according to \cite[Lemma 14.1]{s6}). If $\mathfrak{X}=\mathfrak{N}$ is the class of all nilpotent groups, then $\mathrm{Z}_\mathfrak{N}(G)$ is the hypercenter $\mathrm{Z}_\infty(G)$ of $G$.

\section{The proof of Theorems \ref{thm1} and \ref{thm2}}

Recall \cite[Proposition 1(8)]{Hauck2003} that if $G=G_1\dots G_n$ is the product of pairwise permutable and $\mathfrak{N}$-connected subgroups, then  $[G_i, \prod_{j=1, j\neq i}^nG_j]\leq \mathrm{Z}_\infty(G)$ for any $i\in \{1,\dots, n\}$. According to \cite[Lemma 4.2.12]{PFG} if $G=G_1\dots G_n$ is the product of totally permutable  subgroups, then  $[G_i, \prod_{j=1, j\neq i}^nG_j]\leq \mathrm{Z}_\mathfrak{U}(G)$ for any $i\in \{1,\dots, n\}$ where $\mathfrak{U}$ stands for the class of all supersoluble groups. These observations lead us to the following definition.

\begin{definition}
We say that $G$ is the product of subgroups $G_1, G_2\dots, G_n$ with $\mathfrak{F}$-hypercentral condition for commutators if $G=G_1\dots G_n$,  $G_iG_j$ is a subgroup of $G$ for every $i, j\in\{1, \dots, n\}$ and $[G_i, \prod_{j=1, j\neq i}^nG_j]\leq \mathrm{Z}_\mathfrak{F}(G)$ for any $i\in \{1,\dots, n\}$.
\end{definition}

The main property of products with $\mathfrak{F}$-hypercentral condition for commutators  is

\begin{lemma}\label{lemma1}
  Let $\mathfrak{F}$ be a hereditary formation with $\mathfrak{N}\subseteq \mathfrak{F}$. If a group $G$ is the product of subgroups $G_1,\dots, G_n$ with $\mathfrak{F}$-hypercentral condition for commutators, then
 $$G/\mathrm{Z}_\mathfrak{F}(G)\simeq G_1/\mathrm{Z}_\mathfrak{F}(G_1)\times\dots\times G_n/\mathrm{Z}_\mathfrak{F}(G_n).$$
\end{lemma}

\begin{proof}
$(a)$
 $\overline{H}_i=G_i\mathrm{Z}_\mathfrak{F}(G)/\mathrm{Z}_\mathfrak{F}(G) \cap (\prod_{j=1, j\neq i}^nG_j)\mathrm{Z}_\mathfrak{F}(G)/\mathrm{Z}_\mathfrak{F}(G)\simeq 1$ for any $i\in\{1,\dots,n\}$.

  Since $G$ satisfies  $\mathfrak{F}$-hypercentral condition for commutators, we see that every element of $\overline{H}_i$ commutes with every element of  $G_i\mathrm{Z}_\mathfrak{F}(G)/\mathrm{Z}_\mathfrak{F}(G)$ and $(\prod_{j=1, j\neq i}^nG_j)\mathrm{Z}_\mathfrak{F}(G)/\mathrm{Z}_\mathfrak{F}(G)$. Hence it commutes with every element of $G/\mathrm{Z}_\mathfrak{F}(G)$. Therefore $\overline{H}_i\leq \mathrm{Z}(G/\mathrm{Z}_\mathfrak{F}(G))$. From $\mathfrak{N}\subseteq\mathfrak{F}$ it follows that  $\mathrm{Z}(G/\mathrm{Z}_\mathfrak{F}(G))\leq \mathrm{Z}_\mathfrak{F}(G/\mathrm{Z}_\mathfrak{F}(G))\simeq 1$. Thus $\overline{H}_i\simeq 1$.

$(b)$ $G_i\mathrm{Z}_\mathfrak{F}(G)/\mathrm{Z}_\mathfrak{F}(G)\trianglelefteq G/\mathrm{Z}_\mathfrak{F}(G)$ for any $i\in\{1,\dots,n\}$.

Since $G$ satisfies  $\mathfrak{F}$-hypercentral condition for commutators, we see that every element of  $G_i\mathrm{Z}_\mathfrak{F}(G)/\mathrm{Z}_\mathfrak{F}(G)$ commutes with every element of $(\prod_{j=1, j\neq i}^nG_j)\mathrm{Z}_\mathfrak{F}(G)/\mathrm{Z}_\mathfrak{F}(G)$. Now from $G=G_1\dots G_n$ it follows that $G_i\mathrm{Z}_\mathfrak{F}(G)/\mathrm{Z}_\mathfrak{F}(G)\trianglelefteq G/\mathrm{Z}_\mathfrak{F}(G)$.

$(c)$ $G/\mathrm{Z}_\mathfrak{F}(G)\simeq G_1/\mathrm{Z}_\mathfrak{F}(G_1)\times\dots\times G_n/\mathrm{Z}_\mathfrak{F}(G_n)$.

Now from $(a)$ and $(b)$ it follows that   $$G/\mathrm{Z}_\mathfrak{F}(G)= G_1\mathrm{Z}_\mathfrak{F}(G)/\mathrm{Z}_\mathfrak{F}(G)\times\dots\times G_n\mathrm{Z}_\mathfrak{F}(G)/\mathrm{Z}_\mathfrak{F}(G).$$
Note that every $\mathfrak{F}$-central chief factor of  $G_i\mathrm{Z}_\mathfrak{F}(G)/\mathrm{Z}_\mathfrak{F}(G)$ is an $\mathfrak{F}$-central chief factor of  $G/\mathrm{Z}_\mathfrak{F}(G)$. From $\mathrm{Z}_\mathfrak{F}(G/\mathrm{Z}_\mathfrak{F}(G))\simeq 1$ it follows that $$\mathrm{Z}_\mathfrak{F}(G_i\mathrm{Z}_\mathfrak{F}(G)/\mathrm{Z}_\mathfrak{F}(G))\simeq \mathrm{Z}_\mathfrak{F}(G_i/(G_i\cap \mathrm{Z}_\mathfrak{F}(G)))\simeq 1.$$ Since $\mathfrak{F}$ is hereditary, we see that $G_i\cap \mathrm{Z}_\mathfrak{F}(G)\leq \mathrm{Z}_\mathfrak{F}(G_i)$ by \cite[Lemma 2.4(iii)]{Aivazidis2021}. Now from $\mathrm{Z}_\mathfrak{F}(G_i/(G_i\cap \mathrm{Z}_\mathfrak{F}(G)))\simeq 1$ it follows that $G_i\cap \mathrm{Z}_\mathfrak{F}(G)= \mathrm{Z}_\mathfrak{F}(G_i)$. Thus
$$G/\mathrm{Z}_\mathfrak{F}(G)\simeq G_1/(G_1\cap \mathrm{Z}_\mathfrak{F}(G))\times\dots\times G_n/(G_n\cap \mathrm{Z}_\mathfrak{F}(G_n))\simeq G_1/\mathrm{Z}_\mathfrak{F}(G_1)\times\dots\times G_n/\mathrm{Z}_\mathfrak{F}(G_n).$$
Lemma is proved.\end{proof}

Denote by $\Gamma(\mathfrak{F})_{|G}$   the induced subgraph of $\Gamma(\mathfrak{F})$ on $\pi(G)$.

\begin{lemma}\label{lemma2}
  Let $\mathfrak{F}$ be a hereditary formation with $\mathfrak{N}\subseteq \mathfrak{F}$, $\Gamma\in\{\Gamma_s, \Gamma_{Nc}, \Gamma_H\}$ and $G$ be a group. Then
  $$\Gamma(G/\mathrm{Z}_\mathfrak{F}(G))\subseteq\Gamma(G)\subseteq\Gamma(G/\mathrm{Z}_\mathfrak{F}(G))\cup\Gamma(\mathfrak{F})_{|G}. $$
\end{lemma}

\begin{proof}
  From 2 of Lemma \ref{lem1} it follows that $\Gamma(G/\mathrm{Z}_\mathfrak{F}(G))\subseteq\Gamma(G)$. Assume that there is a group $G$ with $\Gamma(G)\not\subseteq\Gamma(G/\mathrm{Z}_\mathfrak{F}(G))\cup\Gamma(\mathfrak{F})_{|G}$.  Note that $V(\Gamma(G))=V(\Gamma(G/\mathrm{Z}_\mathfrak{F}(G))\cup\Gamma(\mathfrak{F})_{|G})$. Hence there is $(p, q)\in E(\Gamma(G))\setminus E(\Gamma(G/\mathrm{Z}_\mathfrak{F}(G))\cup\Gamma(\mathfrak{F})_{|G})$.

  Let $\Gamma=\Gamma_{Nc}$. It means there is a Schmidt $(p, q)$-subgroup $H$ of $G$ with $H\not\in \mathfrak{F}$. From $\Gamma_{Nc}(G/\mathrm{Z}_\mathfrak{F}(G))\subseteq\Gamma_{Nc}(G)$ it follows that $H\mathrm{Z}_\mathfrak{F}(G)/\mathrm{Z}_\mathfrak{F}(G)\simeq H/H\cap\mathrm{Z}_\mathfrak{F}(G) $ is nilpotent. Since $\mathfrak{F}$ is hereditary, $H\cap \mathrm{Z}_\mathfrak{F}(G)\leq \mathrm{Z}_\mathfrak{F}(H)$ by \cite[Lemma 2.4(iii)]{Aivazidis2021}. From $\mathfrak{N}\subseteq\mathfrak{F}$ it follows that $H\mathrm{Z}_\mathfrak{F}(G)\in\mathfrak{F}$. Therefore $H\in\mathfrak{F}$, a contradiction.  Thus $\Gamma_{Nc}(G)\subseteq\Gamma_{Nc}(G/\mathrm{Z}_\mathfrak{F}(G))\cup\Gamma_{Nc}(\mathfrak{F})_{|G}$.

 Let $\Gamma=\Gamma_{s}$. Then there is an element $x$ of $G$ which induces an automorphisms of order $q^\alpha$ on a Sylow $p$-subgroup $P$ of $G$. WLOG we may assume that $x$ is a $q$-element of $G$. Note that $x\mathrm{Z}_\mathfrak{F}(G)$ acts trivially on a Sylow $p$-subgroup $P\mathrm{Z}_\mathfrak{F}(G)/\mathrm{Z}_\mathfrak{F}(G)$ of $G/\mathrm{Z}_\mathfrak{F}(G)$. Hence $P\langle x\rangle \mathrm{Z}_\mathfrak{F}(G)/\mathrm{Z}_\mathfrak{F}(G)$ is a nilpotent group. By analogy with the previous paragraph, $P\langle x\rangle \mathrm{Z}_\mathfrak{F}(G)\in\mathfrak{F}$. Hence $(p, q)\in\Gamma_s(\mathfrak{F})$, a contradiction.

Let $\Gamma=\Gamma_{H}$. From \cite[Proposition 2.3(1)]{VM} it follows that there is a chief factor $H/K$ of $G$ below $\mathrm{Z}_\mathfrak{F}(G)$ with  $p\in\pi(H/K)$ and $q\in \pi(G/C_G(H/K))$. From $(H/K)\rtimes G/C_G(H/K)\in\mathfrak{F}$ it follows that $(p, q)\in\Gamma_H(\mathfrak{F})$, a contradiction.
Thus $\Gamma_{H}(G)\subseteq\Gamma_{H}(G/\mathrm{Z}_\mathfrak{F}(G))\cup\Gamma_{H}(\mathfrak{F})_{|G}$.
\end{proof}

The main result of this section is

\begin{theorem}\label{thm}
  Let $\mathfrak{F}$ be a hereditary formation with $\mathfrak{N}\subseteq \mathfrak{F}$ and $\Gamma\in\{\Gamma_s, \Gamma_{Nc}, \Gamma_H\}$. If a group $G$ is the product of subgroups $G_1,\dots, G_n$ with $\mathfrak{F}$-hypercentral condition for commutators, then

  $$\Gamma(G)\subseteq \bigcup_{i=1}^n\Gamma(G_i)\cup\Gamma(\mathfrak{F})_{|G}. $$
\end{theorem}

\begin{proof}
  From Lemma \ref{lemma1} it follows that
  $$G/\mathrm{Z}_\mathfrak{F}(G)\simeq G_1/\mathrm{Z}_\mathfrak{F}(G_1)\times\dots\times G_n/\mathrm{Z}_\mathfrak{F}(G_n).$$
  Now $\Gamma(G/\mathrm{Z}_\mathfrak{F}(G))=\cup_{i=1}^n\Gamma(G_i/\mathrm{Z}_\mathfrak{F}(G_i))$ by 5 of Lemma \ref{lem1}. Note that $\Gamma(\mathfrak{F})_{|G_i}\subseteq\Gamma(\mathfrak{F})_{|G}$.  Therefore by Lemma \ref{lemma2}
\begin{multline*}
  \Gamma(G)\cup\Gamma(\mathfrak{F})_{|G}=
  \Gamma(G/\mathrm{Z}_\mathfrak{F}(G))\cup\Gamma(\mathfrak{F})_{|G}=\\
\bigcup_{i=1}^n\Gamma(G_i/\mathrm{Z}_\mathfrak{F}(G_i))\cup\Gamma(\mathfrak{F})_{|G}
=\bigcup_{i=1}^n(\Gamma(G_i/\mathrm{Z}_\mathfrak{F}(G_i))\cup\Gamma(\mathfrak{F})_{|G_i}) \cup\Gamma(\mathfrak{F})_{|G}\\
=\bigcup_{i=1}^n(\Gamma(G_i)\cup\Gamma(\mathfrak{F})_{|G_i}) \cup\Gamma(\mathfrak{F})_{|G}
=\bigcup_{i=1}^n\Gamma(G_i) \cup\Gamma(\mathfrak{F})_{|G}\end{multline*}
Thus $\Gamma(G)\subseteq \cup_{i=1}^n\Gamma(G_i)\cup\Gamma(\mathfrak{F})_{|G} $.\end{proof}

\begin{lemma}\label{equal}
  If a group $G$ has a Sylow tower, then $\Gamma_s(G)=\Gamma_{Nc}(G)=\Gamma_H(G)$.
\end{lemma}

\begin{proof}
  According to \cite[Proposition 2.4]{VM} $\Gamma_s(G)\subseteq\Gamma_{Nc}(G)\subseteq\Gamma_H(G)$ for any group $G$. Hence we need to prove only  that $\Gamma_s(G)=\Gamma_H(G)$ for a Sylow tower group $G$. Assume the contrary, let a Sylow tower group $G$ be a minimal order counterexample. Since $V(\Gamma_H(G))=V(\Gamma_s(G))=\pi(G)$, we see that there is $(p, q)\in E(\Gamma_H(G))\setminus E(\Gamma_s(G))$. Since $G$ has a Sylow tower, we see that $\mathrm{O}_{p',p}(G)$ contains all Sylow $p$-subgroups of $G$. Therefore $p\neq q$.

  Let $N$ be a minimal normal subgroup of $G$. Recall that the class of Sylow tower groups is closed under taking epimorphic images. Therefore $G/N$ is a Sylow tower group. From $|G/N|<|G|$ and our assumption it follows that $\Gamma_s(G/N)=\Gamma_H(G/N)$, in particular $(p, q)\not\in E(\Gamma_H(G))$.     If $G$ has two minimal normal subgroups $N_1$ and $N_2$, then $\Gamma_s(G)=\Gamma_s(G/N_1)\cup\Gamma_s(G/N_2)=\Gamma_H(G/N_1)\cup\Gamma_H(G/N_2)=\Gamma_H(G)$ by Lemma \ref{lem1}, a contradiction. Thus  $G$ has the unique minimal normal subgroup $N$. Since $G$ is soluble, $N$ is an $r$-group for some prime $r$. If $r\neq p$, then $\mathrm{O}_{p', p}(G/N)=\mathrm{O}_{p',p}(G)/N$. Hence $(p, q)\in E(\Gamma_H(G/N))$, a contradiction. Now $r=p$. Since $G$ is a Sylow tower group with the unique minimal   normal subgroup $N$, we see that a Sylow $p$-subgroup $P$ of $G$ is normal in $G$. Let $Q$ be a Sylow $q$-subgroup of   $G$. Then $Q\leq N_G(P)$. From $(p, q)\not\in E(\Gamma_s(G))$ it follows that $Q\leq C_G(P)$. Now $Q\leq C_G(H/K)$ where $H/K$ is a chief $p$-factor of $G$. Recall that $\mathrm{O}_{p', p}(G)$ is the intersection of centralizers of all chief $p$-factors of $G$. So $Q\leq \mathrm{O}_{p', p}(G)$. Thus $(p, q)\not\in E(\Gamma_H(G))$, the final contradiction.
\end{proof}

\begin{proof}[Proof of Theorem \ref{thm1}]
  From Lemma \ref{equal} it follows that $\Gamma_s(\mathfrak{N})=\Gamma_{Nc}(\mathfrak{N})=\Gamma_H(\mathfrak{N})$. Note that $V(\Gamma_H(\mathfrak{N}))=\mathbb{P}$ and $E(\Gamma_H(\mathfrak{N}))=\emptyset$.
  Let $\Gamma\in\{\Gamma_s, \Gamma_{Nc}, \Gamma_H\}$. Now from the proof of Theorem  \ref{thm} it follows that
 $$ \Gamma(G)= \Gamma(G)\cup\Gamma(\mathfrak{N})_{|G}
=\bigcup_{i=1}^n\Gamma(G_i) \cup\Gamma(\mathfrak{N})_{|G}=\bigcup_{i=1}^n\Gamma(G_i).$$
Theorem \ref{thm1} is proved.\end{proof}

\begin{proof}[Proof of Theorem \ref{thm2}]
  From Lemma \ref{equal} it follows that $\Gamma_s(\mathfrak{U})=\Gamma_{Nc}(\mathfrak{U})=\Gamma_H(\mathfrak{U})$. Note that $V(\Gamma_H(\mathfrak{U}))=\mathbb{P}$.  It is well known that a group $G$ is supersoluble iff $G/\mathrm{O}_{p',p}(G)$ is abelian of exponent dividing $p-1$. Hence $E(\Gamma_H(\mathfrak{U}))\subseteq\{(p, q)\mid q\in\pi(p-1)\}$. On the other hand if $q\in\pi(p-1)$, then a cyclic group $Z_p$ of order $p$ has a power automorphism of order $q$. Now $H=Z_p\rtimes Z_q$ is supersoluble and $q\in\pi(H/\mathrm{O}_{p',p}(H))$. Thus $E(\Gamma_H(\mathfrak{U}))=\{(p, q)\mid q\in\pi(p-1)\}$.   Now Theorem \ref{thm2} directly follows from Theorem \ref{thm}.\end{proof}

\section{The proof of Theorems \ref{mut} and \ref{thm4}}

%If $A$ and $B$ are normal subgroups, then $\Gamma_{Nc}(G)=\Gamma_{Nc}(A)\cup\Gamma_{Nc}(B)$ by \cite[Theorem 2.7]{VM}.

 We need the following lemma in the proof of Theorem \ref{mut}.

\begin{lemma}\label{Centralizer}
  Let $P$ be a Sylow $p$-subgroup of a Schmidt $(p, q)$-subgroup   $S$ of a group $G$. If $A$ is a subgroup of $G$ with $P\leq A$ and $G=AC_G(P)$, then $A$ contains a Schmidt $(p, q)$-subgroup.   \end{lemma}

\begin{proof}
  Let $Q$ be a Sylow $q$-subgroup of $S$, then $Q=\langle x\rangle$ is cyclic. Since $G=AC_G(P)=C_G(P)A$, there exist $y\in A$ and $z\in C_G(P)$ with $x=zy$. Now $P=P^x=P^y$. It means that $P\trianglelefteq P\langle y\rangle\leq A$.    Assume that $A$ does not contain a Schmidt $(p, q)$-group. Now $P\mathrm{O}_q(\langle y\rangle)$ is a $p$-closed $\{p, q\}$-group without Schmidt $(p, q)$-subgroups. It means that $P\mathrm{O}_q(\langle y\rangle)$ is nilpotent. Hence $\langle y_1\rangle= \mathrm{O}_q(\langle y\rangle)\leq C_G(P)$. Let $y_2=\mathrm{O}_{q'}(\langle y\rangle)$. So $y=y_1y_2$   It is well known that $C_G(P)\trianglelefteq N_G(P)$. Note that $x, y\in N_G(P)$. Now $\langle x\rangle C_G(P)/C_G(P)$ is a non-trivial $q$-group. From the other hand  $\langle x\rangle C_G(P)/C_G(P)=\langle zy_1y_2\rangle C_G(P)/C_G(P)=\langle y_2\rangle C_G(P)/C_G(P)$ is a $q'$-group, a contradiction. \end{proof}

\begin{proof}[Proof of Theorem \ref{mut}]
\textbf{\emph{Let prove that }}$\Gamma_{Nc}(G)\subseteq \Gamma_{Nc}(A)\cup \Gamma_{Nc}(B)\cup\Gamma(A, B)$.
Note that  $\Gamma_{Nc}(A)\cup\Gamma_{Nc}(B)\subseteq \Gamma_{Nc}(G)$ by Lemma \ref{lem1}.

Assume that $\Gamma_{Nc}(G)\subseteq \Gamma_{Nc}(A)\cup \Gamma_{Nc}(B)\cup\Gamma(A, B)$ is false. Let chose a minimal order group $G$ such that $G$ is a mutually permutable product of subgroups $ A$ and $B$ and  $\Gamma_{Nc}(G)\not\subseteq \Gamma_{Nc}(A)\cup \Gamma_{Nc}(B)\cup\Gamma(A, B)$.
It means that there is $(p, q)\not\in E(\Gamma_{Nc}(A)\cup \Gamma_{Nc}(B)\cup\Gamma(A, B))$ such that $(p, q)\in E(\Gamma_{Nc}(G))$. Therefore $G$ has a Schmidt $(p, q)$-subgroup $S$.

Since $A_GB_G\neq 1$ by \cite[Theorem 4.3.11]{PFG}, WLOG we may assume that $A_G$ contains a minimal normal subgroup $N$  of $G$.

 Now $G/N=(A/N)(BN/N)$ is a mutually permutable product of subgroups $A/N$ and $ BN/N$ by \cite[Lemma 4.1.10]{PFG}. Hence      $\Gamma_{Nc}(G/N)\subseteq \Gamma_{Nc}(A/N)\cup \Gamma_{Nc}(BN/N)\cup\Gamma(A/N, BN/N)$. Note that $BN/N\simeq B/(B\cap N)$. It means that $\Gamma_{Nc}(A/N)\subseteq \Gamma_{Nc}(A)$ and $\Gamma_{Nc}(BN/N)=\Gamma_{Nc}(B/(B\cap N))\subseteq \Gamma_{Nc}(B)$ by Lemma \ref{lem1}. By the definition of $\Gamma(A,B)$ we see that $$\Gamma(A/N, BN/N)=\Gamma(A/N, B/(B\cap N))\subseteq\Gamma(A,B).$$ Now $\Gamma_{Nc}(G/N)\subseteq \Gamma_{Nc}(A)\cup \Gamma_{Nc}(B)\cup\Gamma(A, B)$. So $(p, q)\not\in E(\Gamma_{Nc}(G/N))$. Hence $SN/N$ is not a Schmidt group. Therefore $S\cap N$ contains a Sylow $p$-subgroup $P_0$ of $S$. Denote a Sylow $q$-subgroup of $S$ by $Q_0$.

Assume that $N\leq A\cap B$. There exist Sylow $q$-subgroups $Q, Q_1$ and $Q_2$ of $G, A$ and $B$ respectively such that $Q=Q_1Q_2$.  Note that there is $x\in G$ with $Q_0\leq Q^x$. Now $S\leq NQ^x=(NQ_1)^x(NQ_2)^x$. Let $T=NQ^x$,  $H=(NQ_1)^x$ and $K=(NQ_2)^x$.  From $\Gamma_{Nc}(H)=\Gamma_{Nc}(NQ_1)\subseteq\Gamma_{Nc}(A)$ and $\Gamma_{Nc}(K)=\Gamma_{Nc}(NQ_2)\subseteq\Gamma_{Nc}(B)$ it follows that $(p, q)\not\in E(\Gamma_{Nc}(H)\cup \Gamma_{Nc}(K))$.
Let $P$ be a Sylow $p$-subgroup of $N$ with $P_0\leq P$. By Frattini's Argument $H=NN_H(P)$. So there is a $q$-subgroup $Q_3$ of $ N_H(P)$ with $H=NQ_3$. Note that $PQ_3$ is a $p$-closed $\{p, q\}$-group without Schmidt $(p, q)$-subgroups. It means that $PQ_3$ is nilpotent. Hence $ Q_3\leq C_H(P)$. Therefore $H=NC_H(P)$. Similar arguments show that $K=NC_K(P)$. So $T=NC_T(P)$. Now $N$ contains a Schmidt $(p, q)$-group by Lemma \ref{Centralizer}, a contradiction.

Assume now that $N\not\leq A\cap B$. It means that $N\cap B=1$ by \cite[Lemma 4.3.3(4)]{PFG}.  Suppose that $B\leq C_G(N)$.
Now $A$ has a Schmidt  $(p, q)$-group by  Lemma \ref{Centralizer}, a contradiction.
Thus $B\not\leq C_G(N)$.
In this case $N$ is cyclic and $A\leq C_G(N)$ by \cite[Lemma 4.3.3(5)]{PFG}.
Since $G=AB=C_G(N)(NB)$, we see that $NB$ contains a Schmidt  $(p, q)$-group by  Lemma \ref{Centralizer}. Hence $q\in\pi(NB/C_{NB}(N))\subseteq\pi(B)$.
Since $N$ is a cyclic $p$-group, we see that $ NB/C_{NB}(N)\simeq G/C_G(N)$ is abelian of exponent dividing $p-1$. Therefore $(p, q)\in E(\Gamma(A, B))$, the final contradiction.

\textbf{\emph{Let prove that}} $\Gamma_{H}(G)\subseteq \Gamma_{H}(A)\cup \Gamma_{H}(B)\cup\Gamma(A, B)\cup\{(p, p)\mid p\in\pi(G)\}$.
Note that  $\Gamma_{H}(A)\cup \Gamma_{H}(B)\subseteq \Gamma_{H}(G)$ by Lemma \ref{lem1}.

Assume that $\Gamma_{H}(G)\subseteq \Gamma_{H}(A)\cup \Gamma_{H}(B)\cup\Gamma(A, B)\cup\{(p, p)\mid p\in\pi(G)\}$ is false. Let chose a minimal order group $G$ such that $G$ is a mutually permutable product of subgroups $ A$ and $B$ and  $\Gamma_{H}(G)\not\subseteq \Gamma_{Nc}(A)\cup \Gamma_{Nc}(B)\cup\Gamma(A, B)\cup\{(p, p)\mid p\in\pi(G)\}$.
It means that there is $(p, q)\not\in E(\Gamma_{H}(A)\cup \Gamma_{H}(B)\cup\Gamma(A, B)\cup\{(p, p)\mid p\in\pi(G)\})$ such that $(p, q)\in E(\Gamma_{H}(G))$. In particular, $p\neq q$.

WLOG we may assume that $A$ contains a minimal normal subgroup $N$  of $G$ by \cite[Theorem 4.3.11]{PFG}. Now $G/N=(A/N)(BN/N)$ is a mutually permutable product of subgroups $A/N$ and $ BN/N$ by \cite[Lemma 4.1.10]{PFG}. Hence      $\Gamma_{H}(G/N)\subseteq \Gamma_{H}(A/N)\cup \Gamma_{H}(BN/N)\cup\Gamma(A/N, BN/N)\cup\{(p, p)\mid p\in\pi(G/N)\}$. Note that  $\Gamma_{H}(A/N)\subseteq \Gamma_{H}(A)$,  $\Gamma_{H}(BN/N)\subseteq \Gamma_{H}(B)$ by Lemma \ref{lem1}, and $\Gamma(A/N, BN/N)\subseteq\Gamma(A,B)$. Now $\Gamma_H(G/N)\subseteq \Gamma_{H}(A)\cup \Gamma_{H}(B)\cup\Gamma(A, B)\cup\{(p, p)\mid p\in\pi(G)\}$. So $(p, q)\not\in E(\Gamma_{H}(G/N))$. From 3 of Lemma \ref{lem1} and our assumption it follows that $N$ must be the unique minimal normal subgroup of $G$.
If $\Phi(G)\not\simeq 1$, then similar arguments show that $(p, q)\not\in E(\Gamma_H(G/\Phi(G)))$. Note that $\Gamma_H(G/\Phi(G))=\Gamma_H(G)$ by \cite[Theorem 2.7]{VM}, a contradiction. So $\Phi(G)=1$. Thus $G$ is a primitive group with $C_G(N)\leq N$.

 If $N$ is a $p'$-group, then $\mathrm{O}_{p',p}(G/N)=\mathrm{O}_{p',p}(G)/N$ and $G/\mathrm{O}_{p',p}(G)\simeq (G/N)/\mathrm{O}_{p',p}(G/N)$. Hence $(p, q)\in E(\Gamma_{H}(G/N))$, a contradiction. Now $p\in\pi(N)$. Therefore $\mathrm{O}_{p'}(G)=1$.
 Assume that $N\leq A\cap B$. Hence $\mathrm{O}_{p'}(A)=\mathrm{O}_{p'}(B)=1$.  Now $(\pi(A)\setminus\{p\})\subseteq\pi(A/\mathrm{O}_{p',p}(A))$ and $(\pi(B)\setminus\{p\})\subseteq\pi(B/\mathrm{O}_{p',p}(B))$. It means that $(p, q)\in E(\Gamma_H(A)\cup\Gamma_H(B))$, a contradiction. Therefore $N\not\leq A\cap B$. It means that $N\cap B=1$ by \cite[Lemma 4.3.3(4)]{PFG}.

 Now either $A\leq C_G(N)$ or $B\leq C_G(N)$ by \cite[Lemma 4.3.3(5)]{PFG}. If $B\leq C_G(N)$, then from $C_G(N)\leq N\leq A$ it follows that $A=G$, and $(p, q)\in E(\Gamma_H(A))$, a contradiction.
Thus $B\not\leq C_G(N)$.
In this case $N$ is cyclic and $A\leq C_G(N)$ by \cite[Lemma 4.3.3(5)]{PFG}. Hence $N\leq A\leq C_G(N)\leq N$. Thus $N=C_G(N)=A$ is a cyclic group of order $p$.
In this case $G/N$ is an abelian group of exponent dividing $p-1$.
From $\mathrm{O}_{p'}(G)=1$   it follows that is a $\mathrm{O}_{p', p}(G)=N$. Therefore $\pi(G/\mathrm{O}_{p', p}(G))\subseteq\pi(p-1)$. Hence $q\in\pi(p-1)$.  Thus $(p, q)\in E(\Gamma(A, B))$, the final contradiction.
\end{proof}

\begin{proof}[Proof of Corollary \ref{Shemetkov}]
  Let $\mathfrak{F}$ be a hereditary formation with the Shemetkov property. Then $\mathfrak{F}=(G\mid\Gamma_{Nc}(G)\subseteq\Gamma_{Nc}(\mathfrak{F}))$ by \cite[Theorem 4.4]{VM}. 
  Assume that $\mathfrak{F}$ is closed under taking mutually permutable products. Let $G$ be a supersoluble Schmidt $\pi(\mathfrak{F})$-group. Then $G/\Phi(G)$ is a mutually permutable product of groups $Z_p$ and $Z_q$  of orders $p$ and $q$ for some $p, q\in\pi(\mathfrak{F})$ with $q\in\pi(p-1)$. Hence $G/\Phi(G)\in\mathfrak{F}$. Since the class of soluble $\mathfrak{F}$-groups is saturated \cite[Corollary 6.4.5]{BallesterBollinches2006}, $G\in\mathfrak{F}$. Thus $\mathfrak{F}$ contains every supersoluble Schmidt $\pi(\mathfrak{F})$-group.
  
  Assume now that $\mathfrak{F}$ contains every supersoluble Schmidt $\pi(\mathfrak{F})$-group. Hence $(p, q)\in E(\Gamma_{Nc}(\mathfrak{F}))$ for every $p, q\in\pi(\mathfrak{F})$ with $q\in\pi(p-1)$. Now if $G=AB$ is a mutually permutable product of $\mathfrak{F}$-subgroups $A$ and $B$, then $\Gamma_{Nc}(G)\subseteq\Gamma_{Nc}(\mathfrak{F})$ by Theorem \ref{mut}. Hence $G\in\mathfrak{F}$.         Thus $G$ is closed under taking mutually permutable products. 
\end{proof}

\begin{proof}[Proof of Corollary \ref{Beidleman}]
  Let $\pi$ be a $p$-special set of primes and $\mathfrak{F}$ be the class of normal extensions of $p$-groups by $\pi$-groups. Then $p\not\in\pi$. Hence $\mathfrak{F}$ is the formation of all $p$-closed $\pi(\mathfrak{F})$-groups. Recall that an $s$-critical group for the class of all $p$-closed groups     is a Schmidt $(q, p)$-group for some prime $q$. Hence $\mathfrak{F}$ is the formation with the Shemetkov property. Note that $\mathfrak{F}$ contains every supersoluble $\pi(\mathfrak{F})$-group. Thus $\mathfrak{F}$ is closed under taking mutually permutable products by Corollary \ref{Shemetkov}.
\end{proof}

We need the following Lemma in the prove of Theorem \ref{thm4}.

\begin{lemma}\label{ro2}
  Let $G$ be a soluble group and $n\geq 3$. If $G=G_1\dots G_n$ is a product of pairwise permutable subgroups $G_1,\dots, G_n$, then $$\Gamma_{Nc}(G)=\bigcup_{1\leq i<j\leq n} \Gamma_{Nc}(G_iG_j).$$
\end{lemma}

\begin{proof}
  Assume that $n=3$, then $G=(G_1G_2)(G_1G_3)=(G_1G_2)(G_2G_3)=(G_1G_3)(G_2G_3)$. Now $\Gamma_{Nc}(G)=\Gamma_{Nc}(G_1G_2)\cup\Gamma_{Nc}(G_1G_3)\cup\Gamma_{Nc}(G_2G_3)$ by \cite[Theorem 7.1(1)]{VM}. Assume that we prove Lemma \ref{ro2} for all $n$ with $3\leq n\leq k$, let prove it for $n=k+1$. Let $H_l=\prod_{j=1, j\neq l}^n$. Then by our assumption  $\Gamma_{Nc}(H_l)=\bigcup_{1\leq i<j\leq n, i,j\neq l} \Gamma_{Nc}(G_iG_j)$. From $G=H_1H_2=H_1H_3=H_2H_3$ and  \cite[Theorem 7.1(1)]{VM} it follows that  $$\Gamma_{Nc}(G)=\Gamma_{Nc}(H_1)\cup\Gamma_{Nc}(H_2)\cup\Gamma_{Nc}(H_3)=\bigcup_{1\leq i<j\leq k+1} \Gamma_{Nc}(G_iG_j).$$
  Now Lemma \ref{ro2} follows from the mathematical induction principle.
\end{proof}

\begin{proof}[Proof of Theorem \ref{thm4}]
Let a group $G$ be a product of pairwise mutually permutable soluble subgroups $G_1,\dots, G_n$.
From \cite[Theorem 4.1.14]{PFG} it follows that $G$ is soluble. Now  $$\Gamma_{Nc}(G)=\bigcup_{1\leq i<j\leq n} \Gamma_{Nc}(G_iG_j).$$
According to Theorem \ref{mut}  $\Gamma_{Nc}(G_iG_j)\subseteq \Gamma_{Nc}(G_i)\cup\Gamma_{Nc}(G_j)\cup\Gamma(G_i, G_j)$. Thus
 $$\Gamma_{Nc}(G)\subseteq\bigcup_{1\leq i\leq n}\Gamma_{Nc}(G_i)\cup\bigcup_{1\leq i, j\leq n, i\neq j} \Gamma(G_i, G_j).$$

 From $\Gamma_{Nc}(G)\subseteq \Gamma_H(G)$ for every group $G$ and  \cite[Lemma 3]{Murashka2019} it follows that if $(p, q)\in E(\Gamma_H(G))\setminus E(\Gamma_{Nc}(G))$ for a soluble group $G$, then $p=q$. Thus
 $$\Gamma_{H}(G)\subseteq\bigcup_{1\leq i\leq n}\Gamma_{H}(G_i)\cup\bigcup_{1\leq i, j\leq n, i\neq j} \Gamma(G_i, G_j)\cup\{(p, p)\mid p\in\pi(G)\}.$$
 Theorem \ref{thm4} is proved.
\end{proof}

\begin{proof}[Proof of Corollary \ref{cor41}]
It is clear that the class $\mathfrak{F}$ of all groups whose Schmidt subgroups are supersoluble is a hereditary formation with the Shemetkov property and $\Gamma_{Nc}(\mathfrak{F})=\Gamma_{Nc}(\mathfrak{U})=\{(p, q)\mid q\in\pi(p-1)\}$. Note that $\Gamma_{Nc}(\mathfrak{F})$ does not have cycles. Hence every $\mathfrak{F}$-group has a Sylow tower by \cite[Theorem 6.2(b)]{VM}. In particular, $G_1,\dots, G_n$ are soluble.   Now  $G=G_1\dots G_n$ is a product of mutually permutable soluble $\mathfrak{F}$-subgroups $G_1,\dots, G_n$. Therefore $\Gamma_{Nc}(G)\subseteq\Gamma_{Nc}(\mathfrak{F})$ by Theorem \ref{thm4}. Thus $G\in\mathfrak{F}$ by \cite[Theorem 4.4]{VM}.
\end{proof}

\section{Final remarks and open questions}

From 4 of Lemma \ref{lem1} it follows that if a group $G=AB$ is the product of its normal subgroups $A$ and $B$,   then $\Gamma(G)=\Gamma(A)\cup\Gamma(B)$ where $\Gamma\in\{\Gamma_H, \Gamma_{Nc}\}$. Note that $S_4$ is the product of its normal subgroups $S_4$ and $A_4$ and $\Gamma_s(S_4)\cup\Gamma_s(A_4)\neq\Gamma_s(S_4)$. Nevertheless the following question seems interesting.

\begin{pr}
  If a group $G=AB$ is the product of its normal subgroups $A$ and $B$.   Is $\Gamma_s(G)\subseteq \Gamma_s(A)\cup\Gamma_s(B)$?
\end{pr}

Through $ \overline{\Gamma}$ here we denote an undirected graph on the same vertex set as $ \Gamma$ in which two vertices are connected by the edge if they are connected in $\Gamma$.
In the proves of \cite{DAniello2007, Kazarin2011} the Sylow graph was considered as a directed one but in \cite{Kazarin2011} it was defined  as an undirected one. Therefore the graph $\overline{\Gamma}_s$ seems interesting. Moreover

\begin{proposition}
  If a soluble group $G=AB$ is the product of its normal subgroups $A$ and $B$, then $\overline{\Gamma}_s(G)=\overline{\Gamma}_s(A)\cup\overline{\Gamma}_s(B)$.
\end{proposition}

\begin{proof}
  From \cite[Theorem 4.2(2)]{Murashka2021} it follows that $\overline{\Gamma}_s(H)=\overline{\Gamma}_{Nc}(H)$ for any soluble group $H$. Now   $\overline{\Gamma}_s(G)=\overline{\Gamma}_s(A)\cup\overline{\Gamma}_s(B)$ follows from 4 of Lemma \ref{lem1}.
\end{proof}

Note \cite[Proof of Theorem 4.2]{Murashka2021} that there are groups $H$ with $\overline{\Gamma}_s(H)\neq\overline{\Gamma}_{Nc}(H)$. That is why we ask

\begin{pr}
  If a group $G=AB$ is the product of its normal subgroups $A$ and $B$.   Is $\overline{\Gamma}_s(G)=\overline{\Gamma}_s(A)\cup\overline{\Gamma}_s(B)$?
\end{pr}

In Theorem \ref{mut} only $N$-critical and Hawkes graph of mutually permutable product was described. What can be said about the Sylow graph of mutually permutable\,product?\,For\,example

\begin{pr}
  If a group $G=AB$ is the product of mutually permutable subgroups $A$ and $B$.   Is $\Gamma_s(G)\subseteq \Gamma_s(A)\cup\Gamma_s(B)\cup\Gamma(A, B)$?
\end{pr}

The proof of Theorem \ref{mut} is based on the properties of mutually permutable products of 2 subgroups. The analogues of these properties for products of more than 2 subgroups are not known now. That is why we use some properties of $N$-critical graph of a soluble group (see Lemma \ref{ro2}) to prove Theorem   \ref{thm4}. Hence we have the following 2 questions.

\begin{pr}
  Does the conclusion of Theorem \ref{thm4} hold for the product of pairwise mutually permutable subgroups $G_1,\dots, G_n$?
\end{pr}

\begin{pr}
  Let $G$ be a  group and $n\geq 3$. If $G=G_1\dots G_n$ is the product of pairwise permutable subgroups $G_1,\dots, G_n$, then is $$\Gamma_{Nc}(G)=\bigcup_{1\leq i<j\leq n} \Gamma_{Nc}(G_iG_j)?$$
\end{pr}

%\begin{corollary}Let $G=AB$ be a mutually permutable product of its subgroups $A$ and $B$,   $\pi_1=\cup_{p\in\pi(A)}\pi(p-1)$  and $\pi_2=\cup_{p\in\pi(B)}\pi(p-1)$. If $A$ does not have cyclic chief $\pi_2$-factors and $B$ does not have cyclic chief $\pi_1$-factors, then $\Gamma_{Nc}(A)\cup\Gamma_{Nc}(B)= \Gamma_{Nc}(G)$.\end{corollary}

{\small\bibliographystyle{plain}
\bibliography{forest}}

\end{document}